\documentclass[10pt]{article}

\usepackage{latexsym, amscd, amsfonts, eucal, mathrsfs, amsmath, amssymb, amsthm, xr, makeidx, stmaryrd, bigints}
\usepackage{tikz-cd}
\usepackage{bookmark}
\usepackage[a4paper, left=20mm, right=20mm, top=25mm, bottom=20mm]{geometry}
\usepackage{indentfirst}
\usepackage{enumitem}
\usepackage{enumitem}
\usepackage{verbatim}
\usepackage{upgreek}
\usepackage{textcomp}
\usepackage{dsfont}
\usepackage{mathtools}
\usepackage{commutative-diagrams}
\usepackage{hyperref, cleveref}
\usepackage{mdframed}
\hypersetup{colorlinks=true,linkcolor=orange,citecolor=orange,urlcolor=orange}

\newcommand{\X}{\mathbb{X}}

\newcommand{\Hom}{\operatorname{Hom}}

\newcommand{\Z}{\mathbb{Z}}
\newcommand{\Q}{\mathbb{Q}}

\newcommand{\Gal}{\operatorname{Gal}}

\newcommand{\C}{\mathbb{C}}
\newcommand{\cont}{\mathrm{Cont}}

\newcommand{\Ext}{{\mathcal{E}xt}}

\newcommand{\coker}{\operatorname{coker}}
\renewcommand{\cont}{\mathrm{cont}}

\newcommand{\id}{\mathrm{id}}

\newcommand{\fib}{\operatorname{fib}}
\newcommand{\cofib}{\operatorname{cofib}}
\newcommand{\Stab}{\operatorname{Stab}}
\newcommand{\nequiv}{\not\equiv}

\newcommand*{\modcat}[1][]{{\mathrm{Mod}_{\if\relax\detokenize{#1}\relax\else#1,\fi p}^{\mathrm{disc}}}}

\newcommand*{\dmodcat}[1][]{{D^+\!\bigl(\mathrm{Mod}_{\if\relax\detokenize{#1}\relax\else#1,\fi p}^{\mathrm{disc}}\bigr)}}

\newcommand*{\chaincat}[1][]{{\operatorname{Ch}^+\!\bigl(\mathrm{Mod}_{\if\relax\detokenize{#1}\relax\else#1,\fi p}^{\mathrm{disc}}\bigr)}}

\theoremstyle{plain}
\newtheorem{thm}{Theorem}[section]
\newtheorem{lem}[thm]{Lemma}
\newtheorem{prop}[thm]{Proposition}
\newtheorem{cor}[thm]{Corollary}

\theoremstyle{definition}
\newtheorem{defn}[thm]{Definition}

\theoremstyle{remark}
\newtheorem{rmk}[thm]{Remark}

\begin{document}
\title{\textbf{On the cohomology of negative Tate twists\\via cyclotomic descent}}
\author{Taewan Kim, Seunghun Ryu}
\maketitle
\begin{abstract}
We show that the Galois cohomology of negative Tate twists can be organized by a single universal cyclotomic complex over the cyclotomic tower of $\mathbb{Q}$. Using cyclotomic descent and Teichm\"uller branch decomposition, we prove that a negative twist contributes only on the corresponding branch and is recovered by specializing the Iwasawa variable at a single point; equivalently, it is computed as the fiber of $\gamma-u^{-m}$, or $T=u^{-m}-1$ in Iwasawa coordinates. In the case $\mathbb{Q}_p/\mathbb{Z}_p$, this gives explicit descriptions of $H^1$ and $H^2$ in terms of the quotient and torsion of the $S$-ramified Iwasawa module.
\end{abstract}

\tableofcontents
\bigskip

\section{Notation and Background}
Throughout the paper we adopt the following conventions unless stated otherwise.

\begin{itemize}
\setlength\itemsep{0.1em}
    \item $F=\Q$, $F_n:=\Q(\mu_{p^n})$ and $F_\infty:=\Q(\mu_{p^\infty})=\bigcup_{n=1}^\infty \Q(\mu_{p^n})$.
    \item $p$ : an arbitrary \emph{odd} prime.
    \item $S=\{p,\infty\}$ : a finite set of places of $\Q$.
    \item $F_S=\Q_S$ : the maximal extension of $\Q$ unramified outside of $S$.
    \item $G:=G(F_S/F)=\Gal(\Q_S/\Q)$ and $H:=G(F_S/F_\infty)$.
    \item $\Gamma:=G(F_\infty/F)\simeq \Z_p^\times$, $\Gamma_1=G(F_\infty/F_1)$ and $\Delta:=G(F_1/F)$.
    \item $\Lambda:=\Lambda(\Gamma_1)=\Z_p[[\Gamma_1]]=\varprojlim_U \Z_p[\Gamma_1/U]$ : the Iwasawa algebra.
\end{itemize}
There is a short exact sequence
\[
1\longrightarrow H\longrightarrow G\longrightarrow \Gamma\longrightarrow 1.
\]
Let
\[
\chi_{\mathrm{cyc}}:G\longrightarrow \Z_p^\times
\]
be the cyclotomic character. Then $H=\ker(\chi_{\mathrm{cyc}})$, and the induced character
\[
\chi_\Gamma:\Gamma\xrightarrow{\sim}\Z_p^\times
\]
is the tautological character. Since $p$ is odd, one has $\Z_p^\times\simeq \mu_{p-1}\times (1+p\Z_p)$. We think of $\mu_{p-1}$ and $1+pZ_p$ as subsets of $\Z_p^\times$. $\Delta$ and $\Gamma_1$ are the image of $\mu_{p-1}$ and $1+p\Z_p$ in $\Gamma$, respectively; so we have
\[
\Gamma= \Delta\times \Gamma_1,
\]
\[
\Delta\cong \mu_{p-1},
\qquad
\Gamma_1\cong 1+p\Z_p\cong \Z_p.
\]
Let
\[
\omega:\Z_p^\times\longrightarrow \mu_{p-1}\subset \Z_p^\times,
\]
\[
\langle-\rangle:\Z_p^\times\longrightarrow 1+pZ_p\subset \Z_p^\times
\]
where $\omega$ is the Teichm\"uller character of the reduction modulo $p$ and $\langle x\rangle=\omega^{-1}(x)\chi_\Gamma(x)$ is the pro-$p$ part.
Also define
\[
\omega_\Gamma:=\omega\circ\chi_\Gamma:\Gamma\longrightarrow \mu_{p-1}\subset \Z_p^\times,
\]
\[
\langle\cdot\rangle_\Gamma:=\langle\cdot\rangle\circ\chi_\Gamma:\Gamma\longrightarrow 1+pZ_p\subset \Z_p^\times.
\]
We can view $\omega_\Gamma$ as a $\Z_p^\times$-valued character.
Denote $\omega_\Delta:=\omega_\Gamma|_\Delta:\Delta\to\mu_{p-1}$.
Under these notations, we can write $\chi_\Gamma=\omega_\Gamma\cdot \langle-\rangle_\Gamma$ and $\chi_\Gamma|_\Delta=\omega_\Delta$.
Fix a topological generator $\gamma$ of $\Gamma_1$, and define $u:=\langle\gamma\rangle_\Gamma\in 1+p\Z_p$.
Since $\omega_\Gamma(\gamma)=1$, we have $\chi_\Gamma(\gamma)=\omega_\Gamma(\gamma)\cdot\langle\gamma\rangle_\Gamma=u$.
Note that the homomorphisms $\Delta\cong \mu_{p-1}\to\Z_p^\times$ are given precisely by the $\omega_\Delta^i$, where $i\in \Z/(p-1)\Z$. Thus for any character $\vartheta:\Gamma\to\Z_p^\times$, there is an integer $i\in\Z/(p-1)\Z$ such that $\vartheta|_\Delta=\omega_\Delta^i$.

\begin{defn}
For a profinite group $K$, denote by $\modcat[K]$ the abelian category of discrete $p$-primary abelian groups with continuous $K$-action. This is a Grothendieck abelian category with enough injectives. Let us denote by $\dmodcat[K]$ its bounded-below derived $\infty$-category. Concretely, one may realize it as the differential graded nerve of the dg category of bounded-below complexes of injective objects. This is a stable $\infty$-category. Its homotopy category is the usual bounded-below derived category. When $K=1$, write $\modcat[]$ for the abelian category of discrete $p$-primary abelian groups. 
\end{defn}

Let $M$ be a discrete $G$-module. We denote $\Stab_M(x)\subset G$ by the stabilizer of $x\in M$ with respect to the $G$-action of $M$.
Note that, every $p$-primary abelian group carries a natural $\Z_p$-module structure, so multiplication by any $p$-adic unit is well-defined.

\medskip

If $K'\triangleleft K$ is a closed normal subgroup and $\overline K=K/K'$, then the functor of $K'$-invariants
\[
(-)^{K'}:\modcat[K]
\longrightarrow
\modcat[\overline K]
\]
is left exact. Its right derived functor is denoted
\[
R\Gamma(K',-):
\dmodcat[K]
\longrightarrow
\dmodcat[\overline K].
\]
For $N\in \modcat[K]$,
\[
R\Gamma(K',N)\in \dmodcat[\overline K]
\]
denotes the derived object attached to $N$ placed in degree $0$. In particular, when $K'=K$, one has
\[
R\Gamma(K,N)\in \dmodcat[],
\]
and
\[
H^q(K,N):=H^q\bigl(R\Gamma(K,N)\bigr)
\]
is the group cohomology of $K$ with coefficients in $N$. 
Note that $H^{-q}(K, N)=0$ for $q>0$, because we work on bounded below complexes.

\medskip

The diagram below
\begin{center}
\begin{tikzcd}[row sep=1.5em, column sep=4em]
    & {\modcat[\Gamma]} \arrow[rd, "{(-)^\Gamma}"] & \\
    {\modcat[G]} \arrow[rr, "{(-)^G}"'] \arrow[ru, "{(-)^H}"] & & {\modcat[]}
\end{tikzcd}
\end{center}
yields a diagram
\begin{center}
\begin{tikzcd}[row sep=1.5em, column sep=4em]
    & {\dmodcat[\Gamma]} \arrow[rd, "{R\Gamma(\Gamma, -)}"] & \\
    {\dmodcat[G]} \arrow[rr, "{R\Gamma(G, -)}"'] \arrow[ru, "{R\Gamma(H, -)}"] & & {\dmodcat[]}
\end{tikzcd}
\end{center}
The derived $\infty$-categories and the right derived functors displayed above are used frequently in this paper.
On the other hand, since the action-forgetful functors are exact, we have
\begin{center}
\begin{tikzcd}[row sep=1.5em, column sep=4em]
    & {\dmodcat[\Gamma]} \arrow[rd, "{\operatorname{Forget}}"] & \\
    {\dmodcat[G]} \arrow[rr, "{\operatorname{Forget}}"'] \arrow[ru, "{\operatorname{Forget}}"] & & {\dmodcat[]}
\end{tikzcd}
\end{center}

All fibers and cofibers are formed in these stable $\infty$-categories. (see \cite[Definition 1.1.1.6]{ha})
For any morphism $f:X\to Y$ in a stable $\infty$-category, $\fib(f)$ is an object which makes the diagram below to be pullback:
\[
\begin{CD}
    {\fib(f)} @>>> X\\
    @VVV @VVfV\\
    0 @>>> Y
\end{CD}
\]
and $\cofib(f)$ is an object that makes the diagram below to be pushout:
\[
\begin{CD}
    X @>f>> Y\\
    @VVV @VVV\\
    0 @>>> {\cofib(f)}
\end{CD}
\]
One has a canonical equivalence (cf.\ proof of \cite[Lemma 1.1.3.3]{ha})
\[
\fib(f)\simeq \cofib(f)[-1],
\]
where $[-1]$ means applying loop functor $\Omega$.

\begin{rmk}[cf.\ {\cite[Proposition 5.2.4 (ii)]{nsw}}]
For every profinite group $K$, the category $\modcat[K]$ is Grothendieck abelian and has enough injectives; for discrete $K$-modules, continuous cohomology agrees with the right derived functors of invariants; and every exact endofunctor of such an abelian category extends degreewise to bounded-below complexes and preserves quasi-isomorphisms. If, in addition, it preserves injectives, then it is represented degreewise on the injective model and therefore induces an exact endofunctor on the bounded-below derived $\infty$-category. In particular, every exact autoequivalence preserves injectives, because it has an exact quasi-inverse and hence is both left and right adjoint to an exact functor.
\end{rmk}

\begin{prop}\label{prop: inflation_adjunction}
Let $1\to K'\to K\to \overline K\to 1$ be a short exact sequence of profinite groups and assume that $K'$ be a closed normal subgroup of $K$. Then the inflation functor
\[
\operatorname{Inf}_{\overline K}^{K}:
\modcat[\overline K]
\longrightarrow
\modcat[K]
\]
obtained by pulling back the action along $K\twoheadrightarrow \overline K$ is exact and left adjoint to
\[
(-)^{K'}:
\modcat[K]
\longrightarrow
\modcat[\overline K].
\]
In particular, the functor $(-)^{K'}$ preserves injectives.
\end{prop}

\begin{proof}
Inflation is exact because it does not change the underlying abelian group. Let $N\in \modcat[\overline K]$ and $M\in \modcat[K]$. A $K$-equivariant map
\[
f:\operatorname{Inf}_{\overline K}^{K}(N)\longrightarrow M
\]
has image contained in $M^{K'}$, because $K'$ acts trivially on the source. Conversely, every $\overline K$-equivariant map $N\longrightarrow M^{K'}$ is $K$-equivariant after inflation. Hence there is a natural bijection
\[
\Hom_{K}\bigl(\operatorname{Inf}_{\overline K}^{K}(N),M\bigr)
\cong
\Hom_{\overline K}\bigl(N,M^{K'}\bigr).
\]
Thus $\operatorname{Inf}_{\overline K}^{K}$ is left adjoint to $(-)^{K'}$. Since $(-)^{K'}$ is a right adjoint to the inflation functor, which is exact and preserves injectives, the final statement follows.
\end{proof}

\section{General Cyclotomic descent}

In this section, we assume that $A\in \modcat[G]$ is an $p$-primary discrete abelian group endowed with the trivial $G$-action. For a $K$-module $M$, we denote $\rho_M(k):M\to M$ by the action map induced by $k \in K$.

\begin{defn}\label{def: weight_twist}
    Let $\vartheta:\Gamma\to \Z_p^\times$ be a continuous character. For $M\in \modcat[\Gamma]$, twisting by $\vartheta$ defines an exact autoequivalence
    \[
    (-)\{\vartheta\}:\modcat[\Gamma]
    \longrightarrow
    \modcat[\Gamma]
    \]
    \[
    M\longmapsto M\{\vartheta\}
    \]
    by keeping the same underlying abelian group and rescaling the $\Gamma$-action:
    \[
    \rho_{M\{\vartheta\}}(\delta):=\vartheta(\delta)\cdot\rho_M(\delta)
    \qquad
    (\delta\in \Gamma).
    \]
\end{defn}

\begin{defn}\label{def: character_twist}
    Let $\vartheta:\Gamma\to \Z_p^\times$ be a continuous character. Twinsting by $\vartheta$ defines an exact autoequivalence
    \[
    (-)(\vartheta):\modcat[G]
    \longrightarrow
    \modcat[G]
    \]
    \[
    M\longmapsto M(\vartheta)
    \]
    by keeping the same underlying abelian group and twisting the $G$-action:
    \[
    \rho_{M(\vartheta)}(g)(x):=\vartheta(\bar g)\cdot\rho_{M}(g)(x),
    \]
    where $\bar g$ is the image of $g\in G$ in $\Gamma$.
\end{defn}

The above actions are continuous. Let us prove this for $M(\vartheta)$; the proof for $M\{\vartheta\}$ is almost the same. Let $x\in M$. We can find a natural number $n$ and an open subgroup $U\subset G$ such that $x$ is killed by $p^n$ and fixed by $U$, because $M$ is $p$-primary and the action map is continuous. Then a subgroup
\[
V=U\cap \ker\bigl(G\twoheadrightarrow \Gamma\xrightarrow{\vartheta}\Z_p^\times\to (\Z/p^n\Z)^\times\bigr)
\]
fixes $x$ for the twisted action, so $V\subset\Stab_{M(\vartheta)}(x)$.
The maps in the kernel are all continuous, and $\{1\}\subset (\Z/p^n\Z)^\times$ is open by the discrete topology. Hence the kernel is open as a preimage of an open set, so $V$ is an open subgroup of $G$. It implies that $\Stab_{M(\vartheta)}(x)$ is open and the twisted action is continuous.

\medskip

$(-)\{\vartheta\}$ is exact since it keeps the underlying abelian group. Furthermore, $\vartheta^{-1}$ gives its quasi-inverse $(-)\{\vartheta^{-1}\}$; hence $(-)\{\vartheta\}$ is autoequivalence on $\modcat[\Gamma]$.
Applying these degreewise defines an autoequivalence on $\chaincat[\Gamma]$. Since the quasi-inverse $(-)\{\vartheta^{-1}\}$ is also exact, the functor $(-)\{\vartheta\}$ preserves injectives and hence induces an exact autoequivalence, again denoted $(-)\{\vartheta\}$ on $\dmodcat[\Gamma]$. 
For the same reason, $(-)(\vartheta)$ is an exact autoequivalence on $\modcat[G]$ and also we extend an exact autoequivalence $(-)(\vartheta)$ on $\dmodcat[G]$.

\medskip

For every integer $m$, one may realize that the action of $A(\chi_\Gamma^m)$ is the same as the action of the usual Tate twist $A(m)$, because $A$ is endowed with the trivial $G$-action. In this light, for an integer $m$, we denote 
\[
A(m)=A(\chi_\Gamma^m)
\]
and
\[
M^\bullet\{m\}:=M^\bullet\{\chi_\Gamma^m\},
\]
\[
M^\bullet(m):=M^\bullet(\chi_\Gamma^m).
\]
Note that, one has canonical isomorphisms of $H$-modules $A(m)\big|_H\cong A\big|_H$, because $H=\ker(\chi_{\mathrm{cyc}})$.

\begin{defn}[Universal cyclotomic complexes]
For a discrete $G$-module $A$ equipped with trivial $G$-action, define
\[
\X(A):=R\Gamma(H,A)\in \dmodcat[\Gamma].
\]
The $\Gamma$-action on this derived object is the natural one induced by conjugation through the quotient map $G\twoheadrightarrow G/H\simeq \Gamma$.
\end{defn}

\begin{prop}\label{prop: twisting_commutes}
Let $\vartheta:\Gamma\to \Z_p^\times$ be a continuous character and let $M^\bullet\in\chaincat[G]$. Then there is a canonical equivalence 
\[
R\Gamma(H,M^\bullet(\vartheta))\simeq R\Gamma(H,M^\bullet)\{\vartheta\}
\]
in $\dmodcat[\Gamma]$.
\end{prop}
\begin{proof}
Since $\vartheta$ factors through $\Gamma\simeq G/H$, the $H$-action on $M^\bullet(\vartheta)$ is the same as the $H$-action on $M^\bullet$. Thus there is a natural identification of left exact functors
\[
\bigl(M^\bullet(\vartheta)\bigr)^H \cong \bigl((M^\bullet )^H\bigr)\{\vartheta\}
\]
in $\modcat[\Gamma]$. Choose a bounded-below injective resolution $M^\bullet\to I^\bullet$ in $\modcat[G]$. Since $(-)(\vartheta)$ preserves injectives, $M^\bullet(\vartheta)\longrightarrow I^\bullet(\vartheta)$ is an injective resolution of $M^\bullet(\vartheta)$. Taking $H$-invariants termwise and using the preceding identification gives an isomorphism of complexes
\[
\bigl(I^\bullet(\vartheta)\bigr)^H \cong \bigl((I^\bullet)^H\bigr)\{\vartheta\}.
\]
Passing to the derived $\infty$-category yields
\[
R\Gamma(H,M^\bullet(\vartheta))
\simeq
\bigl(I^\bullet(\vartheta)\bigr)^H
\simeq
\bigl((I^\bullet)^H\bigr)\{\vartheta\}
\simeq
R\Gamma(H,M^\bullet)\{\vartheta\}.
\]
\end{proof}

\begin{cor}\label{cor: upstair_identification}
For every integer $m$, there are canonical equivalences in $\dmodcat[\Gamma]$
\[
R\Gamma(H,A(m))\simeq \X(A)\{m\}.
\]
\end{cor}

\begin{proof}
Apply the previous lemma with $M^\bullet=A$ and $\vartheta=\chi_\Gamma^m$. Since $A$ has trivial $G$-action, $A(\chi_\Gamma^m)=A(m)$. Hence
\[
R\Gamma(H,A(m))
\simeq
R\Gamma(H,A)\{m\}
=
\X(A)\{m\}.
\]
\end{proof}

\begin{defn}
For $M^\bullet\in \dmodcat[\Gamma]$ and $m\in \Z$, define the \emph{$m$-th cyclotomic descent} of $M^\bullet$ by
\[
\mathrm{CD}_m(M^\bullet):=R\Gamma(\Gamma,M^\bullet\{m\})
\in \dmodcat[].
\]
\end{defn}

\begin{prop}[Hochschild--Serre]\label{prop: Hochschild_Serre}
Let $K$ be a profinite group and $K'$ its closed normal subgroup. Let
\[
1\longrightarrow K'\longrightarrow K\longrightarrow \overline K\longrightarrow 1
\]
be a short exact sequence of profinite groups. Then for every $M^\bullet\in \dmodcat[K]$, there is a canonical equivalence
\[
R\Gamma(K,M^\bullet)\simeq R\Gamma\bigl(\overline K,R\Gamma(K',M^\bullet)\bigr)
\]
in $\dmodcat[]$.
\end{prop}

\begin{proof}
Choose a bounded-below complex of injective $K$-modules $I^\bullet$ representing $M^\bullet$. Since $(-)^{K'}$ preserves injectives (see \ref{prop: inflation_adjunction}), $(I^\bullet)^{K'}$ is a bounded-below complex of injective $\overline K$-modules. Therefore
\[
R\Gamma\bigl(\overline K,R\Gamma(K',M^\bullet)\bigr)
\simeq
\bigl((I^\bullet)^{K'}\bigr)^{\overline K}
=
(I^\bullet)^K
\simeq
R\Gamma(K,M^\bullet).
\]
This equivalence is functorial in $M^\bullet$.
\end{proof}

\begin{thm}[Cyclotomic descent]\label{thm: cyclotomic_descent}
    Let $\vartheta:\Gamma\to \Z_p^\times$ be a continuous character.
    Then there is a canonical equivalence
    \[
    R\Gamma(G,A(\vartheta))
    \simeq R\Gamma\bigl(\Gamma,\X(A)\{\vartheta\}\bigr).
    \]
    In particular, for every integer $m$, 
    \[
    R\Gamma(G,A(m))
    \simeq R\Gamma\bigl(\Gamma,\X(A)\{m\}\bigr)
    \simeq \mathrm{CD}_m(\X(A)).
    \]
\end{thm}
\begin{proof}
    Applying Proposition \ref{prop: Hochschild_Serre} for the extension
    \[
    1\longrightarrow H\longrightarrow G\longrightarrow \Gamma\longrightarrow 1,
    \]
    one has
    \[
    R\Gamma(G,A(\vartheta))
    \simeq
    R\Gamma\bigl(\Gamma,R\Gamma(H,A(\vartheta))\bigr).
    \]
    By Corollary \ref{cor: upstair_identification},
    \[
    R\Gamma(H,A(\vartheta))\simeq \X(A)\{\vartheta\}.
    \]
    Therefore
    \[
    R\Gamma(G,A(\vartheta))
    \simeq
    R\Gamma\bigl(\Gamma,\X(A)\{\vartheta\}\bigr).
    \]
    If $\vartheta=\chi^m_\Gamma$ for an integer $m$, the latter is equal to $\mathrm{CD}_m(\X(A))$.
\end{proof}

\begin{cor}
For every integer $m$, there is a convergent spectral sequence
\[
E_2^{a,b}
=
H^a\bigl(\Gamma,H^b(H,A)\{m\}\bigr)
\Longrightarrow
H^{a+b}(G,A(m)).
\]
\end{cor}
\begin{proof}
Consider the composite of left exact functors
\[
\modcat[G]
\xrightarrow{(-)^H}
\modcat[\Gamma]
\xrightarrow{(-)^\Gamma}
\modcat[].
\]
By Proposition \ref{prop: inflation_adjunction}, the functor $(-)^H$ preserves injectives. Hence the Grothendieck spectral sequence for this composite yields
\[
E_2^{a,b}
=
H^a\bigl(\Gamma,H^b(H,A(m))\bigr)
\Longrightarrow
H^{a+b}(G,A(m)).
\]
Using Corollary \ref{cor: upstair_identification} and the exactness of twisting, one gets
\[
H^b(H,A(m))
\simeq
h^b(\X(A)\{m\})
\simeq
h^b(\X(A))\{m\}
=
H^b(H,A)\{m\}.
\]
This identifies the $E_2$-page with the displayed one.
\end{proof}

Recall that $\omega_\Delta=\omega_\Gamma|_\Delta:\Delta\to \mu_{p-1}\subset \Z_p^\times$ is the restriction of the Teichm\"uller character to $\Delta$.
\begin{defn}[Teichm\"uller branch idempotents]
For each residue class $j\in \Z/(p-1)\Z$, define
\[
e_j
=
\frac{1}{p-1}
\sum_{\delta\in \Delta}\omega_\Delta^{-j}(\delta)[\delta]
\in \Z_p[\Delta].
\]
Here we identify the codomain $\mu_{p-1}$ of $\omega_\Delta$ with a subset of $\Z_p$.
It can be easily checked that $e_j$ are complete orthogonal idempotents, i.e., $\sum e_j=1$, $e_ie_j=0$ and $e_je_j=e_j$ for $i\neq j$.
If $M$ is a discrete $\Gamma$-module, define the \emph{$e_j$-branch} of $M$ by $M^{(j)}:=e_jM$, the image of the idempotent $e_j$. Then $M$ decomposes functorially as
\[
M=\bigoplus_{j\in \Z/(p-1)\Z} e_jM.
\]
Note that each branch depends only on the residue of $j$ modulo $p-1$. Hence the endofunctor $M\mapsto e_jM$ is exact on $\modcat[\Gamma]$. Moreover $e_jM$ is a direct factor of $M$, so if $M$ is injective then $e_jM$ is injective as well. Hence $e_j$ extends degreewise to bounded-below complexes of injectives and therefore to an exact endofunctor, again denoted $e_j$ on $\chaincat[\Gamma]$ and therefore on $\dmodcat[\Gamma]$.
All branch indices are understood modulo $p-1$.
\end{defn}

\begin{lem}\label{lem: Delta_exact}
    The functor of $\Delta$-invariants
    \[
    (-)^\Delta:\modcat[\Gamma]
    \longrightarrow
    \modcat[\Gamma_1]
    \]
    is exact, and therefore
    \[
    R\Gamma(\Delta,-)\simeq (-)^\Delta
    \]
    in $\dmodcat[\Gamma_1]$.
\end{lem}
\begin{proof}
    Since $|\Delta|=p-1$ is prime to $p$, the averaging operator
    \[
    P_\Delta:=
    \frac{1}{p-1}\sum_{\delta\in \Delta}\delta
    \]
    defines a projection onto $\Delta$-invariants on every discrete $p$-primary $\Gamma$-module. Indeed, if $x\in N^\Delta$, then $P_\Delta\cdot x=\frac{1}{p-1}\sum_{\delta\in\Delta}x=x$. Conversely, for $x=P_\Delta \cdot y$, $\delta\cdot x=\delta\frac{1}{p-1}\sum_{\epsilon\in \Delta}\epsilon y=\frac{1}{p-1}\sum_{\epsilon'\in\Delta}\epsilon'y=x$ for all $\delta\in\Delta$. Hence we have $N^\Delta=P_\Delta N$ for any $\Gamma$-module $N$. Because $\Gamma=\Delta\times\Gamma_1$, every element of $\Gamma_1$ commutes with every element of $\Delta$, so $P_\Delta$ is $\Gamma_1$-equivariant.
    Let $f:N\to N'$ be a surjection in $\modcat[\Gamma]$. For any $P_\Delta \cdot y\in P_\Delta N'$, we can take $x\in N$ such that $f(x)=y$. Then $f^\Delta(P_\Delta\cdot x)=P_\Delta\cdot f(x)=P_\Delta\cdot y$ yields that $f^\Delta$ is surjective; thus $P_\Delta=(-)^\Delta$ is exact.
\end{proof}

Recall that the restriction of any continuous character $\Gamma\to\Z_p^\times$ to $\Delta$ is equal to $\omega_\Delta^m$ for some $m\in\Z/(p-1)\Z$.

\begin{prop}[Branch decomposition of cyclotomic descent]\label{prop: branch_decomposition}
    Let $\vartheta:\Gamma\to\Z_p^\times$ be a continuous character with $\vartheta|_\Delta =\omega_\Delta^m$ for some $m\in\Z/(p-1)\Z$. Let $M^\bullet\in \dmodcat[\Gamma]$. Then
    \[
    R\Gamma\bigl(\Gamma, M^\bullet\{\vartheta\}\bigr)
    \simeq R\Gamma\bigl(\Gamma_1,\,(e_{-m}M^\bullet)\{\vartheta\}\bigr)
    \]
    in $\dmodcat[]$, where $(e_{-m}M^\bullet)\{\vartheta\}$ is viewed as an object of $\dmodcat[\Gamma_1]$ by restricting the twisted $\Gamma$-action along $\Gamma_1\hookrightarrow \Gamma$.
    In particular, for every integer $m$, we have
    \[
    \mathrm{CD}_m(M^\bullet)
    \simeq R\Gamma\bigl(\Gamma_1,\,(e_{-m}M^\bullet)\{m\}\bigr).
    \]
\end{prop}
The above proposition says that only the branch $-m$ modulo $p-1$ in the branch decomposition of $M^\bullet$ contributes to $\mathrm{CD}_m(M^\bullet)$.
\begin{proof}
    Applying Proposition \ref{prop: Hochschild_Serre} to the exact sequence
    \[
    1\longrightarrow \Delta\longrightarrow \Gamma\longrightarrow \Gamma_1\longrightarrow 1
    \]
    gives
    \[
    R\Gamma(\Gamma,M^\bullet\{\vartheta\})
    \simeq
    R\Gamma\bigl(\Gamma_1,R\Gamma(\Delta,M^\bullet\{\vartheta\})\bigr)
    \simeq R\Gamma\bigl(\Gamma_1,(M^\bullet\{\vartheta\})^\Delta\bigr),
    \]
    where the last isomorphism follows from Lemma \ref{lem: Delta_exact}.
    For $\delta\in \Delta$, since $\Delta$ is abelian, one has in
    $\Z_p[\Delta]$
    \[
    [\delta] e_j
    =\frac{1}{p-1}\sum_{\varepsilon\in \Delta}\omega_\Delta^{-j}(\varepsilon)[\delta\varepsilon]
    =\frac{1}{p-1}\sum_{\varepsilon'\in \Delta}\omega_\Delta^{-j}(\delta^{-1}\varepsilon')[\varepsilon']
    =\omega_\Delta^j(\delta)e_j.
    \]
    Hence if $x$ is a homogeneous element of a term of $e_jM^\bullet$, then
    \[
    \delta\cdot x
    =\delta\cdot (e_j\cdot x)
    =(\delta e_j)\cdot x
    =\omega_\Delta^j(\delta)e_j\cdot x
    =\omega_\Delta^j(\delta)\cdot x.
    \]
    Thus $\Delta$ acts by $\omega_\Delta^j$ on the branch $e_jM^\bullet$. After twisting by $\vartheta$, $\delta\in\Delta$ acts on $x\in (e_j M^\bullet)\{\vartheta\}$ by
    \[
    \vartheta(\delta)\cdot(\delta\cdot x)
    =\vartheta|_\Delta(\delta)\cdot \omega_\Delta^j(\delta)\cdot x
    =\omega_\Delta^j(\delta)\cdot \omega_\Delta^m(\delta)\cdot x
    =\omega_\Delta^{j+m}(\delta)\cdot x.
    \]
    Assume that $j+m\not\equiv0\pmod{p-1}$ so $\omega_\Delta^{j+m}$ is nontrivial. Choose $\delta\in \Delta$ such that $\omega_\Delta^{j+m}(\delta)\neq 1$. Then $\omega^{j+m}(\delta)-1\in \Z/p\Z$ is nonzero, hence invertible and lying on $\mu_{p-1}\subset\Z_p^\times$. For all $x\in\bigl((e_j M^\bullet)\{\vartheta\}\bigr)^\Delta$ and $\delta\in \Delta$,
    \[
    0=\delta\cdot x-x=\bigl(\omega^{j+m}(\delta)-1\bigr)\cdot x.
    \]
    Since $\omega^{j+m}(\delta)-1\in \Z_p^\times$, multiplication by $\omega^{j+m}(\delta)-1$ is an automorphism on the $p$-primary group, so $x=0$ and $\bigl((e_j M^\bullet)\{\vartheta\}\bigr)^\Delta=0$. On the other hand, assume that $j+m\equiv0\pmod{p-1}$. Then $\Delta$ acts trivially on $(e_j M^\bullet)\{\vartheta\}$, so we have $\bigl((e_j M^\bullet)\{\vartheta\}\bigr)^\Delta=(e_j M^\bullet)\{\vartheta\}$. Using the branch decomposition  of $M^\bullet$, the $\Delta$-invariant part is precisely
    \[
    (M^\bullet\{\vartheta\})^\Delta
    =\Bigr(\Bigl(\bigoplus_{j\in\Z/(p-1)\Z}e_j M^\bullet\Bigr)\{\vartheta\}\Bigr)^\Delta
    \simeq \bigoplus_{j\in\Z/(p-1)\Z}\bigr((e_j M^\bullet)\{\vartheta\}\bigr)^\Delta
    \simeq (e_{-m}M^\bullet)\{\vartheta\},
    \]
    where the second isomorphism arises from the fact that $(-)\{\vartheta\}$ and $(-)^\Delta$ commute with finite direct sums due to their exactness (see Lemma \ref{lem: Delta_exact} and the discussion following to Definition \ref{def: character_twist}).
    The canonical $\Gamma_1$-action is equal to the restriction of the twisted $\Gamma$-action on $M^\bullet\{m\}$. Hence
    \[
    R\Gamma(\Delta,M^\bullet\{\vartheta\})
    \simeq
    (M^\bullet\{\vartheta\})^\Delta
    \simeq
    (e_{-m}M^\bullet)\{\vartheta\}
    \]
    in $\dmodcat[\Gamma_1]$, and the result follows.
\end{proof}

\begin{cor}\label{cor: specialization}
    Let $\vartheta:\Gamma\to\Z_p^\times$ be a continuous character with $\vartheta|_\Delta =\omega_\Delta^m$. Then
    \[
    R\Gamma(G,A(\vartheta))
    \simeq
    R\Gamma\bigl(\Gamma_1,\,(e_{-m}\X(A))\{\vartheta\}\bigr).
    \]
    In particular, for any integer $m$,
    \[
    R\Gamma(G,A(m))
    \simeq
    R\Gamma\bigl(\Gamma_1,\,(e_{-m}\X(A))\{m\}\bigr).
    \]
\end{cor}
Here again, the twisted complexes on the right are viewed as objects of $D^+\bigl(\modcat[\Gamma_1]\bigr)$
by restriction of the twisted $\Gamma$-action along $\Gamma_1\hookrightarrow \Gamma$.
\begin{proof}
    By Theorem \ref{thm: cyclotomic_descent} and Proposition \ref{prop: branch_decomposition}, we have
    \[
    R\Gamma(G,A(\vartheta))
    \simeq R\Gamma\bigl(\Gamma, \X(A)\{\vartheta\}\bigr)\
    \simeq
    R\Gamma\bigl(\Gamma_1,\,(e_{-m}\X(A))\{\vartheta\}\bigr).
    \]
\end{proof}

\begin{rmk}\label{rmk: coh_dim_Gamma_1}
    By \cite[Corollary 3.5.16 and Proposition 3.5.17]{nsw}, we have $cd_p(\Gamma_1)=cd_p(\Z_p)=1$.
\end{rmk}

\begin{rmk}[cf. {\cite[Proposition 5.2.7]{nsw}}]\label{rmk: iwasawa_EXT}
    Let $K$ be a profinite group and $\Lambda(K)$ its Iwasawa algebra. Then we can express the group cohomology of $M$ using $\Ext$ functor:
    We have a canonical isomorphism
    \[
    H^i(K, M)\cong\Ext^i_{\Lambda(K)}(\Z_p, M)
    \]
    in $\modcat[]$.
\end{rmk}

\begin{lem}[Cohomology of $\Gamma_1$ on a single module]
    Let $M\in \modcat[\Gamma_1]$ and view $M$ as a discrete $\Lambda(\Gamma_1)$-module. Then there is a natural short exact sequence
    \[
    0\longrightarrow M^{\Gamma_1}\longrightarrow M
    \xrightarrow{\gamma-1} M
    \xrightarrow{\;\partial\;} H^1(\Gamma_1, M)\longrightarrow 0,
    \]
    and
    \[
    H^q(\Gamma_1,M)=0
    \qquad (q\ge 2).
    \]
\end{lem}
\begin{proof}
    Consider the following short exact sequence:
    \[
    0\longrightarrow\Lambda(\Gamma_1)\xrightarrow{\gamma-1}\Lambda(\Gamma_1)\xrightarrow{\pi}\Z_p\longrightarrow 0,
    \]
    where $\pi$ is the augmentation map.
    We shall show the sequence is exact. First, the map $\pi$ is clearly surjective. 
    Using the canonical identification $\Lambda(\Gamma_1)\cong\Z_p[T]$ via $\gamma-1\mapsto T$, the map $\gamma-1:\Lambda(\Gamma_1)\to\Lambda(\Gamma_1)$ can be understood by mapping $f(T)\mapsto Tf(T)$, which is obviously injective.
    Consider the usual augmentation ideal $\ker(\pi)\subset\Lambda(\Gamma_1)$, which is generated by $\gamma'-1$ for $\gamma'\in \Gamma_1$. 
    Under the identification, the latter ideal corresponds to the principal ideal $(T)\subset\Z_p[T]$, which is precisely the image of $f(T)\mapsto Tf(T)$.

    \medskip
    
    Applying $\Ext^i_{\Lambda(\Gamma_1)}(-, M)$ to the above short exact sequence, we have a long exact sequence
    \[
    0\longrightarrow \Hom_{\Lambda(\Gamma_1)}(\Z_p, M)\longrightarrow\Hom_{\Lambda(\Gamma_1)}(\Lambda(\Gamma_1), M)\longrightarrow\Hom_{\Lambda(\Gamma_1)}(\Lambda(\Gamma_1), M)
    \]
    \[
    \longrightarrow\Ext^1_{\Lambda(\Gamma_1)}(\Z_p, M)\longrightarrow\Ext^1_{\Lambda(\Gamma_1)}(\Lambda(\Gamma_1), M)\longrightarrow\Ext^1_{\Lambda(\Gamma_1)}(\Lambda(\Gamma_1), M)\longrightarrow0
    \]
    Because $\Gamma_1$ has cohomological $p$-dimension 1 (cf. Remark \ref{rmk: coh_dim_Gamma_1}), the $\Ext^i$ terms in the long exact sequence vanish after $\Ext^1$.
    Since $\Lambda(\Gamma_1)$ is a free $\Lambda(\Gamma_1)$-module, $\Ext^1_{\Lambda(\Gamma_1)}(\Lambda(\Gamma_1), M)$ vanishes.
    Using Remark \ref{rmk: iwasawa_EXT}, we identify $\Ext^1_{\Lambda(\Gamma_1)}(\Z_p, M)$ with $H^1(\Gamma_1, M)$.
    Then the latter sequence becomes
     \[
    0\longrightarrow \Hom_{\Lambda(\Gamma_1)}(\Z_p, M)\longrightarrow\Hom_{\Lambda(\Gamma_1)}(\Lambda(\Gamma_1), M)\longrightarrow\Hom_{\Lambda(\Gamma_1)}(\Lambda(\Gamma_1), M)\longrightarrow H^1(\Gamma_1, M)\longrightarrow0
    \]
    It remains to identify the $\Hom$ sets and the maps between them with the ones arising in the assertion.

\medskip

    Note that $1\in \Z_p$ and $1\in \Lambda(\Gamma_1)$ are generators of each module with respect to $\Lambda(\Gamma_1)$-action.
    Consider $f\in \Hom_{\Lambda(\Gamma_1)}(\Z_p, M)$ and let $f(1)=x\in M$. Since $\Z_p$ is endowed with the trivial $\Gamma_1$-action and $f$ is $\Gamma_1$-equivariant, $x$ is fixed by $\Gamma_1$. Conversely, an element $x\in M^{\Gamma_1}$ determines a unique $f\in\Hom_{\Lambda(\Gamma_1)}(\Z_p, M)$. Therefore we have an isomorphism $\Hom_{\Lambda(\Gamma_1)}(\Z_p, M)\cong M^{\Gamma_1}$ given by $f\mapsto f(1)$. Doing similarly, we have $\Hom_{\Lambda(\Gamma_1)}(\Lambda(\Gamma_1), M)\cong M$. Indeed, $f\in\Hom_{\Lambda(\Gamma_1)}(\Lambda(\Gamma_1), M)$ is uniquely determined by $f(1)=x\in M$, where $x$ can be taken freely.

\medskip
    
    Finally, let $f\in\Hom_{\Lambda(\Gamma_1)}(\Lambda(\Gamma_1), M)$ and $x=f(1)\in M$. Applying $\Lambda(\Gamma_1)\xrightarrow{\gamma-1}\Lambda(\Gamma_1)$, we have $\gamma\cdot x-x\in M$ and this is equal to $(f\circ(\gamma-1))(1)=f((\gamma-1)\cdot 1)=f(\gamma-1)=\gamma\cdot f(1)-f(1)$. Therefore the map $\Lambda(\Gamma_1)\xrightarrow{\gamma-1}\Lambda(\Gamma_1)$ corresponds to $M\xrightarrow{\gamma-1}M$.
    Let $g\in \Hom_{\Lambda(G_1)}(\Z_p, M)$ and let $y=g(1)$. Composing $g$ with the augmentation map, $g$ yields $g\circ\pi\in\Hom_{\Lambda(\Gamma_1)}(\Lambda(\Gamma_1), M)$, corresponding to $g\circ\pi(1)=g(1)=y$. Thus $\pi$ induces the canonical inclusion $M^{\Gamma_1}\hookrightarrow M$. The vanishing $H^q(\Gamma_1, M)=0$ for $q>1$ also follows immediately from Remark \ref{rmk: coh_dim_Gamma_1}.
\end{proof}

\begin{lem}\label{lem: injective_short_ES}
    Let $I^\bullet$ be a bounded-below complex of injective objects in $\modcat[\Gamma_1]$. Then there is a short exact sequence
    \[
    0\longrightarrow (I^\bullet)^{\Gamma_1}\longrightarrow I^\bullet
    \xrightarrow{\gamma-1} I^\bullet
    \longrightarrow 0.
    \]
    in $\chaincat[\Gamma_1]$.
\end{lem}
\begin{proof}
    Fix a degree $q$ and consider an element $x\in I^q$. Choose integers $d,s\ge 0$ such that
    \[
    \gamma^{p^d}x=x,
    \qquad
    p^s x=0,
    \]
    and take $m\ge d+s$. It is possible because $I^q$ is $p$-primary and $\Stab_{I^q}(x)\subset \Gamma_1$ is contained in certain basic open subgroup $\langle\gamma^{p^i}\rangle\subset \Gamma_1$. Then $x\in (I^q)^{U_m}$ and $\mathrm N_m(x)=0$. By Proposition \ref{prop: inflation_adjunction}, $(I^q)^{U_m}$ is injective in $\modcat[G_m]$. Since group cohomology is the right derived functor of invariants, this implies
    \[
    H^1\bigl(G_m,(I^q)^{U_m}\bigr)=0.
    \]
    Using \cite[Proposition 1.7.1]{nsw}, this yields $\ker(\mathrm N_m)=(\overline \gamma_m-1)(I^q)^{U_m}$ and $x\in \ker(\mathrm N_m)=(\bar\gamma_m-1)(I^q)^{U_m}=(\gamma-1)(I^q)^{U_m}\subset (\gamma-1)I^q$. Hence the map $\gamma-1:I^q\to I^q$ is surjective. For every degree $q$, there is a short exact sequence
    \[
    0\longrightarrow (I^q)^{\Gamma_1}\longrightarrow I^q
    \xrightarrow{\gamma-1} I^q
    \longrightarrow 0,
    \]
    and these assemble into a short exact sequence of complexes
    \[
    0\longrightarrow (I^\bullet)^{\Gamma_1}\longrightarrow I^\bullet
    \xrightarrow{\gamma-1} I^\bullet
    \longrightarrow 0,
    \]
    because $\gamma-1$ commutes with boundary maps of the complex, which are $\Gamma_1$-equivariant.
\end{proof}

\begin{prop}\label{prop: Gamma_1_description}
    Let $N^\bullet\in \dmodcat[\Gamma_1]$. Then
    \[
    R\Gamma(\Gamma_1,N^\bullet)
    \simeq
    \fib\bigl(\gamma-1:N^\bullet\to N^\bullet\bigr)
    \simeq
    \cofib\bigl(\gamma-1:N^\bullet\to N^\bullet\bigr)[-1]
    \]
    in $\dmodcat[]$, where on the right $\gamma-1$ denotes the endomorphism of the underlying object obtained from the $\Gamma_1$-action and then forgetting the $\Gamma_1$-action. If $N$ is a single discrete $\Gamma_1$-module placed in degree $0$, this is represented by the two-term complex
    \[
    [N\xrightarrow{\gamma-1}N]
    \]
    in cohomological degrees $0$ and $1$.
\end{prop}
\begin{proof}
    Choose an injective resolution $N^\bullet\longrightarrow I^\bullet$ in $\chaincat[\Gamma_1]$, so we have $R\Gamma(\Gamma_1,N^\bullet)\simeq (I^\bullet)^{\Gamma_1}$. By Lemma \ref{lem: injective_short_ES}, there is a short exact sequence
    \[
    0\longrightarrow (I^\bullet)^{\Gamma_1}\longrightarrow I^\bullet
    \xrightarrow{\gamma-1} I^\bullet\longrightarrow 0.
    \]
    Therefore
    \[
    (I^\bullet)^{\Gamma_1}
    \simeq
    \fib\bigl(\gamma-1:I^\bullet\to I^\bullet\bigr)
    \simeq
    \cofib\bigl(\gamma-1:I^\bullet\to I^\bullet\bigr)[-1].
    \]
    Because fiber and cofiber are exact constructions in the stable $\infty$-category $\dmodcat[]$, the quasi-isomorphism $N^\bullet\to I^\bullet$ induces equivalence
    \[
    \fib\bigl(\gamma-1:N^\bullet\to N^\bullet\bigr)
    \simeq
    \fib\bigl(\gamma-1:I^\bullet\to I^\bullet\bigr).
    \]
    If $N$ is concentrated in degree $0$, then the mapping fiber of
    $\gamma-1:N\to N$ is represented by the two-term complex
    \[
    [N\xrightarrow{\gamma-1}N]
    \]
    in cohomological degrees $0$ and $1$.
\end{proof}

Recall that we choose $\gamma$ to be a topological generator of $\Gamma_1$ so that $\chi_\Gamma(\gamma)=\langle \gamma\rangle=u\in 1+p\Z_p$. All fiber and cofiber descriptions below are taken with respect to this fixed choice of $\gamma$. Since $\Gamma=\Delta\times \Gamma_1$ is abelian, the action of $\gamma\in \Gamma_1$ commutes with the idempotents $e_j$. Hence for every $j\in \Z/(p-1)\Z$, after restricting from $\Gamma$ to $\Gamma_1$ and then forgetting the $\Gamma_1$-action, the action of $\gamma$ induces an endomorphism
\[
\gamma:e_j\X(A)\longrightarrow e_j\X(A)
\]
in $\dmodcat[]$.

\begin{defn}[Cyclotomic cone]
    View $e_j\X(A)$ first as an object of $\dmodcat[\Gamma_1]$ by restricting along $\Gamma_1\hookrightarrow \Gamma$, and then as an object of $\dmodcat[]$ by forgetting the $\Gamma_1$-action. Since $\Gamma=\Delta\times\Gamma_1$ is abelian, the action of $\gamma$ preserves each branch $e_i$, so the action of $\gamma$ induces an endomorphism
    \[
    \gamma:e_j\X(A)\longrightarrow e_j\X(A)
    \]
    in $\dmodcat[]$. Let $\vartheta:\Gamma\to\Z_p^\times$ be a continuous character with $\vartheta|_\Delta=\omega_\Delta^m$ for some $m\in \Z/(p-1)\Z$. Then, multiplication by the scalar $\vartheta(\gamma)\in \Z_p^\times$ likewise defines an endomorphism of the same underlying object. Relative to the fixed topological generator $\gamma$, define
    \begin{align*}
        \mathbb{C}_\vartheta(A)
        :&=
        \fib\Bigl(
        \gamma-\vartheta(\gamma):
        e_{m}\X(A)\longrightarrow e_j\X(A)
        \Bigr)\\
        &\simeq \cofib\Bigl(
        \gamma-\vartheta(\gamma):
        e_{m}\X(A)\longrightarrow e_j\X(A)
        \Bigr)[-1].
    \end{align*}
\end{defn}
If $I^\bullet$ is a bounded-below complex of injective objects in $\modcat[]$ representing the underlying object of $e_j\X(A)$, and if $\widetilde f:I^\bullet\longrightarrow I^\bullet$ is a chain map representing the endomorphism $\gamma-\vartheta(\gamma)$ in $\dmodcat[]$, then the above object is represented by the usual mapping-fiber complex
\[
\fib(\widetilde f)^q = I^q\oplus I^{q-1},
\qquad
d_{\fib(\widetilde f)}(x,y)=(-d_Ix,\; \widetilde f(x)+d_Iy).
\]

\begin{thm}\label{thm: cone_identification}
    Let $A$ be a discrete $G$-module with the trivial action and let $\vartheta:\Gamma\to\Z_p^\times$ be a continuous character with $\vartheta|_\Delta=\omega_\Delta^m$ for some $m\in \Z/(p-1)\Z$. Then there is an equivalence
    \[
    R\Gamma(G,A(\vartheta))\simeq \fib\Bigl(
    \gamma-\vartheta^{-1}(\gamma):
    e_{-m}\X(A)\longrightarrow e_{-m}\X(A)
    \Bigr)=\C_{\vartheta^{-1}}(A)
    \]
    in $\dmodcat[]$.
\end{thm}
\begin{proof}
    By Corollary \ref{cor: specialization} and Proposition \ref{prop: Gamma_1_description},
    \[
    R\Gamma(G,A(\vartheta))
    \simeq
    R\Gamma\bigl(\Gamma_1,\,(e_{-m}\X(A))\{\vartheta\}\bigr)
    \simeq \fib\Bigl(
    \gamma-1:
    (e_{-m}\X(A))\{\vartheta\}
    \longrightarrow
    (e_{-m}\X(A))\{\vartheta\}
    \Bigr).
    \]
    On $(e_{-m}\X(A))\{\vartheta\}$, the action of $\gamma$ is $\vartheta(\gamma)\gamma$. Hence
    \[
    R\Gamma(G,A(\vartheta))
    \simeq
    \fib\Bigl(
    \vartheta(\gamma)\gamma-1:
    e_{-m}\X(A)\longrightarrow e_{-m}\X(A)
    \Bigr).
    \]
    Multiplication by the unit $\vartheta(\gamma)$ is an automorphism of
    $e_{-m}\X(A)$, and the square
    \[
    \begin{CD}
    e_{-m}\X_S @>{\vartheta(\gamma)\gamma-1}>>
    e_{-m}\X_S \\
    @| @V{\times \vartheta^{-1}(\gamma)}V{\sim}V \\
    e_{-m}\X_S @>{\gamma-\vartheta^{-1}(\gamma)}>>
    e_{-m}\X_S
    \end{CD}
    \]
    commutes. Since postcomposition by an equivalence does not change the fiber in a stable $\infty$-category, one gets
    \[
    \fib\bigl(\vartheta(\gamma)\gamma-1\bigr)
    \simeq \fib\bigl(\gamma-\vartheta^{-1}(\gamma)\bigr).
    \]
    Thus
    \[
    R\Gamma(G,A(\vartheta))
    \simeq \fib\bigl(\gamma-\vartheta^{-1}(\gamma)\bigr)=\C_{\vartheta^{-1}}(A).
    \]
\end{proof}

\begin{prop}\label{prop: short_ES}
    For every $q\geq 0$, there is a short exact sequence
    \[
    0\longrightarrow
    \coker\Bigl(
    \gamma-\vartheta^{-1}(\gamma):
    e_{-m}H^{q-1}(H,A)\longrightarrow e_{-m}H^{q-1}(H,A)
    \Bigr)
    \longrightarrow
    H^q(G,A(\vartheta))
    \]
    \[
    \longrightarrow
    \ker\Bigl(
    \gamma-\vartheta^{-1}(\gamma):
    e_{-m}H^q(H,A)\longrightarrow e_{-m}H^q(H,A)
    \Bigr)
    \longrightarrow 0.
    \]
\end{prop}
Note that $H^{-1}(H,A)=0$.
\begin{proof}
Take cohomology of the canonical cofiber sequence
\[
\mathbb C_{\vartheta^{-1}}(A)\longrightarrow e_{-m}\X(A)
\xrightarrow{\gamma-\vartheta^{-1}(\gamma)}
e_{-m}\X(A)
\longrightarrow \mathbb C_{\vartheta^{-1}}(A)[1]
\]
and use the theorem above to identify $\mathbb C_{\vartheta^{-1}}(A)$ with $R\Gamma(G,A(\vartheta))$.
Since $e_{-m}$ is exact, one has
\[
h^q(e_{-m}\X(A))\simeq e_{-m}H^q(H,A).
\]
The resulting long exact sequence yields the displayed short exact sequence
\[
0\to \coker\bigl(\gamma-\vartheta^{-1}(\gamma)\bigr)\to H^q(G,A(\vartheta))\to \ker\bigl(\gamma-\vartheta^{-1}(\gamma)\bigr)\to 0
\]
in each degree $q$, with the first map taken in degree $q-1$ and the second in degree $q$.
\end{proof}

\begin{lem}\label{lem: H0_vanish}
    Let $\vartheta:\Gamma\to\Z_p^\times$ be a continuous character with $\vartheta|_\Delta=\omega_\Delta^m$ for $m\nequiv0\pmod{p-1}$.
    Then, 
    \[
    H^0(G,A(\vartheta))=0.
    \]
\end{lem}
\begin{proof}
    Since $H$ acts trivially on $A$,
    \[
    H^0(H,A)=A.
    \]
    The $\Delta$-action on these $H^0$-groups is trivial, so they lie entirely in the branch $e_0$. Indeed, $e_0\cdot a=a$ for all $a\in A$, so $e_0A=A$. Since $m\not\equiv 0\pmod{p-1}$,
    \[
    e_{-m}H^0(H,A)=0.
    \]
    Applying the previous short exact sequence with $q=0$ gives
    \[
    H^0(G,A(\vartheta))=0.
    \]
\end{proof}

\begin{cor}\label{cor: H1_formula}
    Let $\vartheta:\Gamma\to\Z_p^\times$ be a continuous character with $\vartheta|_\Delta=\omega_\Delta^m$ for $m\nequiv0\pmod{p-1}$.
    Then
    \[
    H^1(G,A(\vartheta))
    \cong
    \ker\Bigl(
    \gamma-\vartheta^{-1}(\gamma):
    e_{-m}H^1(H,A)\longrightarrow e_{-m}H^1(H,A)
    \Bigr).
    \]
\end{cor}
\begin{proof}
    Apply Lemma \ref{lem: H0_vanish} to Proposition \ref{prop: short_ES} for $q=1$.
\end{proof}

\section{Passing to Pontryagin dual}
In this section, we pass to the Pontryagin duals of the previously defined objects to facilitate the application of Iwasawa theory in the next section.

\begin{rmk}
    Let $\mathrm{Mod}^{\mathrm{cpt}}_{K,p}$ denote the category of compact Hausdorff
    abelian pro-$p$ groups with continuous $K$-action. For
    $M \in \modcat[K]$, define its Pontryagin dual by
    \[
    M^\vee := \operatorname{Hom}_{\mathrm{cont}}(M,\Q_p/\Z_p).
    \]
    Since $M$ is discrete, every group homomorphism $M \to \Q_p/\Z_p$ is continuous.
    Endow $M^\vee$ with the compact-open topology and the $K$-action
    \[
    (k\cdot f)(x) := kf(k^{-1}x)=f(k^{-1}x)
    \qquad
    (k\in K,\ f\in M^\vee,\ x\in M).
    \]
    Then $M^\vee$ belongs to $\mathrm{Mod}^{\mathrm{cpt}}_{K,p}$, and Pontryagin duality induces an
    exact contravariant equivalence
    \[
    (-)^\vee :
    \bigl(\modcat[K]\bigr)^{\mathrm{op}}
    \xrightarrow{\sim}
    \mathrm{Mod}^{\mathrm{cpt}}_{K,p}.
    \]
    Its quasi-inverse is again Pontryagin duality on compact modules. For example,
    \[
    (\Q_p/\Z_p)^\vee \cong \Z_p.
    \]
\end{rmk}

\begin{lem}\label{lem: branch_dual_commute}
Let $M$ be a discrete $p$-primary $\Gamma$-module, and let
$m\in \Z/(p-1)\Z$. Then there is a canonical isomorphism of compact $\Gamma$-modules
\[
(e_{-m}M)^\vee \cong e_m(M^\vee).
\]
\end{lem}

\begin{proof}
For $f\in M^\vee$ and $x\in M$, one computes
\[
\begin{aligned}
(e_mf)(x)
&=
\frac{1}{p-1}\sum_{\delta\in\Delta}\omega_\Delta^{-m}(\delta)\,(\delta\cdot f)(x) \\
&=
\frac{1}{p-1}\sum_{\delta\in\Delta}\omega_\Delta^{-m}(\delta)\,f(\delta^{-1}x) \\
&=
\frac{1}{p-1}\sum_{\delta'\in\Delta}\omega_\Delta^{m}(\delta')\,f(\delta'x) \\
&=
f\!\left(
\frac{1}{p-1}\sum_{\delta'\in\Delta}\omega_\Delta^{m}(\delta')\,\delta'x
\right) \\
&= f(e_{-m}x).
\end{aligned}
\]
Therefore $f\in e_m(M^\vee)$ if and only if $f=f\circ e_{-m}$, i.e. if and only if
$f$ factors through the projection $e_{-m}:M\twoheadrightarrow e_{-m}M$. Hence restriction
induces a homomorphism
\[
\Phi:e_m(M^\vee)\longrightarrow (e_{-m}M)^\vee,
\qquad
\Phi(f)=f|_{e_{-m}M}.
\]
Conversely, for $g\in (e_{-m}M)^\vee$, define
\[
\Psi(g)(x):=g(e_{-m}x)
\qquad (x\in M).
\]
By the identity proved above, $\Psi(g)\in e_m(M^\vee)$. Moreover,
\[
\Phi(\Psi(g))(e_{-m}x)=\Psi(g)(e_{-m}x)=g(e_{-m}^2x)=g(e_{-m}x),
\]
so $\Phi\circ\Psi=\id$, and
\[
\Psi(\Phi(f))(x)=\Phi(f)(e_{-m}x)=f(e_{-m}x)=f(x),
\]
so $\Psi\circ\Phi=\id$. Thus $\Phi$ is an isomorphism.
Since $e_{-m}\in \Z_p[\Delta]\subset \Z_p[\Gamma]$ is central, the maps $\Phi$ and $\Psi$
are $\Gamma$-equivariant.
\end{proof}

\begin{rmk}\label{rmk: H1_primary}
    Since $A$ has trivial $H$-action, we have
    \[
    H^1(H,A)=\Hom_{\cont}(H,A).
    \]
    As $H$ is compact and $A$ is discrete, every continuous homomorphism
    $H\to A$ has finite image; since $A$ is $p$-primary, the image is finite $p$-torsion.
    Hence $H^1(H,A)$ is $p$-primary.
\end{rmk}

\begin{prop}
Let $A$ be a discrete $G$-module with trivial action.
Let $\vartheta:\Gamma\to\Z_p^\times$ be a continuous character such that
$\vartheta|_\Delta=\omega_\Delta^m$ for some $m\nequiv 0\pmod{p-1}$. Then
\[
H^1(G,A(\vartheta))^\vee
\simeq
\coker\bigl(\gamma^{-1}-\vartheta^{-1}(\gamma):e_mH^1(H,A)^\vee\longrightarrow e_mH^1(H,A)^\vee\bigr).
\]
\end{prop}

\begin{proof}
By Corollary \ref{cor: H1_formula},
\[
H^1(G,A(\vartheta))
\simeq
\ker\bigl(\gamma-\vartheta^{-1}(\gamma):e_{-m}H^1(H,A)\to e_{-m}H^1(H,A)\bigr).
\]
Since Pontryagin duality is an exact contravariant equivalence, we obtain
\[
H^1(G,A(\vartheta))^\vee
\simeq
\coker\Bigl(
(\gamma-\vartheta^{-1}(\gamma))^\vee:
(e_{-m}H^1(H,A))^\vee
\longrightarrow
(e_{-m}H^1(H,A))^\vee
\Bigr).
\]
By Lemma \ref{lem: branch_dual_commute} and Remark \ref{rmk: H1_primary},
\[
(e_{-m}H^1(H,A))^\vee \simeq e_mH^1(H,A)^\vee.
\]
It remains to identify the dual endomorphism. For $f\in H^1(H,A)^\vee$ and
$x\in H^1(H,A)$,
\[
\bigl((\gamma-\vartheta^{-1}(\gamma))^\vee f\bigr)(x)
=
f(\gamma x)-\vartheta^{-1}(\gamma)f(x)
=
(\gamma^{-1}\cdot f)(x)-\vartheta^{-1}(\gamma)f(x).
\]
Hence
\[
(\gamma-\vartheta^{-1}(\gamma))^\vee
=
\gamma^{-1}-\vartheta^{-1}(\gamma)
\]
on $e_mH^1(H,A)^\vee$, and therefore
\[
H^1(G,A(\vartheta))^\vee
\simeq
\coker\bigl(\gamma^{-1}-\vartheta^{-1}(\gamma):e_mH^1(H,A)^\vee\to e_mH^1(H,A)^\vee\bigr).
\]
\end{proof}

\begin{defn}
    We denote
    \[
    f_m:=\gamma^{-1}-u^{-m}
    \]
    and define
    \[
    X(A):=H^1(H, A)^\vee.
    \]
\end{defn}
\begin{cor}\label{cor: H1_dual}
    Under the previous notations, we have
    \[
    H^1(G, A(\vartheta))^\vee\simeq\frac{e_mX(A)}{f_m\cdot e_mX(A)}.
    \]
\end{cor}

\section{Special case for $\Q_p/\Z_p$}
We apply the results in the previous section to the special case of $A=\Q_p/\Z_p$ with the Tate twist $\vartheta=\chi_\Gamma^m$. In this section, we denote $A=\Q_p/\Z_p$ the $G$-module with the trivial action and $\Lambda:=\Lambda(\Gamma_1)$, the Iwasawa algebra.
This provides the results for $H^i(G, \Q_p/\Z_p(m))$ that were previously known through Iwasawa theory.

\begin{defn}
    Define
    \[
    X_S:=H^1(H, \Q_p/\Z_p)^\vee.
    \]
    This is exactly what we called the \emph{$S$-ramified Iwasawa module}.
    In the same light, define
    \[
    X_{S, n}:=H^1(H, \Z/p^n\Z)^\vee.
    \]
\end{defn}
By \cite[Proposition 11.3.1]{nsw}, $X_S$ is finitely generated $\Lambda$-module.
We have the following characterization of this module $X_S$:
\[
X_S\cong H^{ab}(p)\cong G(F_S(p), F_\infty)^{ab}.
\]
Denote $f_m:=\gamma^{-1}-u^{-m}\in\Lambda(\Gamma_1)$. Under the identification $\Lambda(\Gamma_1)\cong \Z_p[[T]]$ given by $\gamma^{-1}\mapsto T+1$, $f_m$ is equal to $T-u^{-m}+1\in \Z_p[[T]]$.
On the other hand, $e_m$ is not an element of $\Lambda$ as it consists of elements in $\Delta \subset \Gamma$. Nevertheless, as we have previously observed, $\Gamma_1$ acts on each $e_j$-branch due to the fact that $\Delta$ and $\Gamma_1$ commute; hence, $e_mX_S$ remains an Iwasawa $\Lambda$-module.

\begin{rmk}\label{rmk: weak_Leopoldt}
    The weak Leopoldt conjecture holds for the cyclotomic $\Z_p$ extension $F_S$ by \cite[10.3.25]{nsw}. 
    On the other hand, the Ferrero-Washington theorem states that $\mu=\mu(X_S)=0$.
    Therefore, $G(F_S(p)/F_\infty)$ is a free pro-$p$ group by \cite[Theorem 11.3.7]{nsw}.
\end{rmk}
\begin{rmk}\label{rmk: X_S} 
    By \cite[Theorem 10.3.22]{nsw}, $X_S$ has no finite nontrivial submodule, and $rank_{\Lambda(\Gamma)}X_S=r_2=\frac{p-1}{2}$.
    Moreover, applying \cite[Theorem 11.3.2]{nsw} to $k=F_1$, $k_\infty=F_\infty$, $H^2(H, \Q_p/\Z_p)=0$ and $pd_{\Lambda(\Gamma)} X_S\leq 1$.
\end{rmk}

\begin{cor}\label{cor: H_dual}
    Assume that $m\nequiv0\pmod{p-1}$. Then
    \[
    H^1(G, \Q_p/\Z_p(m))^\vee\simeq \frac{e_m X_S}{f_m\cdot e_mX_S},
    \qquad
    H^1(G, \Z/p^n\Z(m))^\vee\simeq \frac{e_m X_{S, n}}{f_m\cdot e_mX_{S, n}}
    \]
    and
    \[
    H^2(G, \Q_p/\Z_p(m))^\vee\simeq e_mX_S[f_m]
    \]
    where $[f_m]$ denotes the $f_m$-torsion part.
\end{cor}
\begin{proof}
    The results for $H^1$ come from Corollary \ref{cor: H1_dual}.
    By \ref{prop: short_ES}, we have 
    \[
    H^2(G, A(m))\simeq \coker\Bigr(\gamma-u^{-m}:e_{-m}H^1(H, A)\longrightarrow e_{-m}H^1(H, A)\Bigr)
    \]
    because the kernel term vanishes by Remark \ref{rmk: X_S}.
\end{proof}
There is a canonical isomorphism
\[
H^i(G, \Q_p/\Z_p(m))^\vee
\simeq \varprojlim_n H^i(G, \Z/p^n\Z(m))^\vee
\]
and this yields a limit description for $H^1$:
\[
\frac{e_m X_S}{f_m\cdot e_mX_S}\simeq \varprojlim_n \frac{e_m X_{S, n}}{f_m\cdot e_mX_{S, n}}.
\]

We shall give some known results for the cohomologies using Corollary \ref{cor: H_dual} with the Iwasawa theory. For simplicity, let us denote $M_m:=e_mX_S$. Since $X_S$ is finitely generated(see Remark \ref{rmk: X_S}), each branch $e_mX_S$ is also finitely generated $\Lambda$-module; hence it admits the standard Iwasawa decomposition into certain elementary $\Lambda$-module $E_m$, which means a homomorphism $M_m\to E_m$ with finite kernel and cokernel. But $M_m$ has no finite submodule, hence the homomorphism has trivial kernel. Writing the cokernel by $C_m$, we have a short exact sequence
\[
0\longrightarrow M_m\longrightarrow E_m\longrightarrow C_m\longrightarrow0,
\]
where
\[
E_m=\Lambda^{r_m}\oplus\bigoplus_{i} \Lambda/p^{k_i}\oplus\bigoplus_{j} \Lambda/F_j^{n_j}.
\]
with $r_m=rank_{\Lambda}M_m$ and the Weirstrass polynomials $F_j$ over $\Z_p$.
\cite[Proposition 11.4.5]{nsw} gives that 
\begin{equation}\label{eqn: r_m}
    r_m=\mathrm{rank}_{\Lambda}e_m X_S=
    \begin{cases}
        1, & \text{if } m \text{ is odd} \\
        0, & \text{if } m \text{ is even} \\
    \end{cases}
\end{equation}
Consider a commutative diagram with exact rows
\begin{center}
    \begin{tikzcd}[column sep=large]
    0 \arrow[r] & M_m \arrow[d, "f_m\cdot"'] \arrow[r] & E_m \arrow[d, "f_m\cdot"'] \arrow[r] & C_m \arrow[d, "f_m\cdot"'] \arrow[r] & 0 \\
    0 \arrow[r] & M_m \arrow[r] & E_m \arrow[r] & C_m \arrow[r] & 0
    \end{tikzcd}
\end{center}
where the objects in the third columns are finite.
Applying the snake lemma yields the 6-term exact sequence
\begin{equation}\label{eqn: 6_ES}
    0\longrightarrow M_m[f_m]\longrightarrow E_m[f_m]\longrightarrow C_m[f_m]
    \longrightarrow \frac{M_m}{f_mM_m}\longrightarrow\frac{E_m}{f_mE_m}\longrightarrow \frac{C_m}{f_mC_m}\longrightarrow 0.
\end{equation}
Note that $f_m=T-(u^{-m}-1)$ is a Weirstrass polynomial of degree 1. By \cite[Lemma 5.3.1]{nsw}, $\Lambda/(f_m)\simeq \Z_p$.
Rewrite the $E_m$ as:
\[
E_m=\Lambda^{r_m}\oplus\bigoplus_{i} \Lambda/p^{l_i}\oplus\bigoplus_{j} \Lambda/F_j^{n_j}\oplus\bigoplus_{k=1}^{t_m} \Lambda/f_m^{m_k}
\]
where the $F_j$ are distinct from $f_m$.
Decomposition of $E_m$ with localizing at $(f_m)$ can be writeen as
\[
(E_m)_{(f_m)}\cong\Lambda_{(f_m)}^{r_m}\oplus\bigoplus_{k=1}^{t_m} \Lambda_{(f_m)}/f_m^{m_k}
\]
with the number $t_m$.
In this setting, let us calculate the ${E_m}/{f_mE_m}$.
\begin{itemize}
\renewcommand\labelitemi{---}
    \item $\Lambda^{r_m}/(f_m)\cong (\Lambda/(f_m))^{r_m}\cong \Z_p^{r_m}$.
    \item $(\Lambda/p^{l_i})/f_m\cong \Lambda/(p^{l_i}, f_m)\cong \Z_p/p^{l_i}$, which is finite module.
    \item $(\Lambda/F_j^{n_j})/f_m\cong \Lambda/(F_j^{n_j}, f_m)\cong \Z_p/\overline{F_j}^{n_j}\cong \Z_p/F_j(u^{-m}-1)^{n_j}$, which is finite module.
    \item $(\Lambda/f_m^{m_k})/f_m\cong \Lambda/(f_m)\cong \Z_p$. $m_k$ does not contribute and it appears $t_m$ times.
\end{itemize}
Therefore, 
\[
\frac{E_m}{f_mE_m}\cong \Z_p^{r_m+t_m}\oplus (\mathrm{finite}).
\]
The exact sequence \ref{eqn: 6_ES} ensures that $M_m/f_mM_m$ differs only finite residue from $E_m/f_mE_m$. Indeed, consider an exact sequence
\[
C_m[f_m]\longrightarrow \frac{M_m}{f_mM_m}\longrightarrow\frac{E_m}{f_mE_m}\longrightarrow \frac{C_m}{f_mC_m}.
\]
Applying the exact functor $-\otimes_{\Z_p}\Q_p$ yields an isomorphism
\[
0\longrightarrow \frac{M_m}{f_mM_m}\otimes_{\Z_p}\Q_p\xrightarrow{\sim}\frac{E_m}{f_mE_m}\otimes_{\Z_p}\Q_p\longrightarrow 0,
\]
because the tensor product with $\Q_p$ annihilates all torsion elements in a finite $\Z_p$-module. Here, $\frac{E_m}{f_mE_m}\otimes_{\Z_p}\Q_p$ is isomorphic to $(\Z_p^{r_m+t_m}\oplus (\mathrm{finite}))\otimes_{\Z_p}\Q_p \cong \Q_p^{r_m+t_m}$. Hence $\frac{M_m}{f_mM_m}\otimes_{\Z_p}\Q_p\cong \Q_p^{r_m+t_m}$, and we can conclude that $M_m/f_mM_m$ is a finitely generated $\Z_p$-module of rank $r_m+t_m$. The structure theorem for finitely generated modules over PID gives
\[
\frac{M_m}{f_mM_m}\cong \Z_p^{r_m+t_m}\oplus (\mathrm{finite}).
\]
Passing to the Pontryagin dual, it yields
\[
H^1(G, \Q_p/\Z_p(m))\cong (\Q_p/\Z_p)^{r_m+t_m}\oplus (\mathrm{finite})
\]
Next, calculate the $E_m[f_m]$. Let $v=-u^{-m}+1$ so that $f_m=T+v$.
\begin{itemize}
\renewcommand\labelitemi{---}
    \item $\Lambda[f_m]=0$
    \item Suppose $g(T)=a_nT^n+\cdots a_0\in \Z_p[[T]]$ satisfies \[f_m(T)g(T)=a_nT^{n+1}+(a_{n-1}+a_nv)T^n+(a_{n-2}+va_{n-1})T^{n-1}+\cdots+(a_0+a_1v)T+a_0v\in p^l\Z_p[[T]].\]
    Then $a_n\in p^l\Z_p$, $a_{n-1}\in p^l\Z_p$, ... , $a_0\in p^l\Z_p$.
    Hence $g\in p^l\Z_p[[T]]$ and $(\Lambda/p^{l_i})[f_m]=0$.
    \item Suppose $g$ satisfies $f_mg\in (F^n)$ where $f_m$ and $F$ are distinct irreducible Weirstrass polynomial. But $F^n|f_mg$ contradicts to the fact that $\Z_p[[T]]$ is UFD and $f_m, F$ are distinct irreducible. Thus $g=0$ and $(\Lambda/F_j^{n_j})[f_m]=0$.
    \item Suppose $g$ satisfies $f_mg\in (f_m^{m_k})$, so $f_m^{m_k-1}|g$. Hence $(\Lambda/f_m^{m_k})[f_m]=f_m^{m_k-1}\Lambda/f_m^{m_k}\cong \Lambda/f_m\cong \Z_p$.
\end{itemize}
Therefore, 
\[
E_m[f_m]\cong \Z_p^{t_m}
\]
From the exact sequence \ref{eqn: 6_ES}, $M_m[f_m]$ is a finite index subgroup of $E_m[f_m]$. But a finite index subgroup of $\Z_p^{t_m}$ is again isomorphic to $\Z_p^{t_m}$; hence,
\[
M_m[f_m]\cong \Z_p^{t_m}
\]
and
\[
H^2(G, \Q_p/\Z_p(m))\cong (\Q_p/\Z_p)^{t_m}.
\]
\begin{cor}
    For $m\nequiv 0\pmod{p-1}$, we have
    \begin{align*}
H^1(G, \Q_p/\Z_p(m)) &\cong (\Q_p/\Z_p)^{r_m+t_m} \oplus (\mathrm{finite}) \\
&\cong \begin{cases}
(\Q_p/\Z_p)^{1+t_m}, & \text{if } m \text{ is odd} \\
(\Q_p/\Z_p)^{t_m}, & \text{if } m \text{ is even}
\end{cases}
\end{align*}
    and
    \[
    H^2(G, \Q_p/\Z_p(m))\cong (\Q_p/\Z_p)^{t_m}.
    \]
\end{cor}

According to \cite{nsw}, we already have $H^2(G, \Q_p/\Z_p(m))=0$ for $m\geq 0$; hence $t_m=0$ for $m\geq 0$.

\begin{rmk}
    By the Iwasawa main conjecture(see \cite[Theorem 11.6.8]{nsw}), $(G_{k, p}(T))=(F_{X_S}(T))$.
    Here $F_{X_S}(T)=\prod_{j}F_j^{n_j}(T)\prod_k f_m^{m_k}(T)$. Note that $f_m(T)$ are degree 1 with zero at $T=u^{-m}-1$ and $F_j$ do not have zeros on $T=u^{-m}-1$. Thus 
    \[
    ord_{T=u^{-m}-1} G_{k, p}(T)
    =ord_{T=u^{-m}-1} F_{X_S}(T)
    =ord_{T=u^{-m}-1} \prod_{k=1}^{t_m} f_m^{m_k}(T)
    =\sum_{k=1}^{t_m}m_k.
    \]
    $t_m=0$ (i.e. $H^2(G, \Q_p/\Z_p(m))=0$) is equivalent to $ord_{T=u^{-m}-1} G_{k, p}(T)=0$, i.e., $G_{k, p}(u^{-m}-1)\neq 0$. But, $G_{k, p}(u^{-m}-1)=(q^{-m}-1)\zeta_{k, p}(m+1)$ since $m\neq 0$. Therefore $t_m=0$ if and only if $\zeta_{k, p}(m+1)\neq 0$.
\end{rmk}

Because the $e_m$-branch depends on $m$ modulo $p-1$, we can write $E_m$ for $0<m<p-1$:
\[
E_m=\Lambda^{r_m}\oplus\bigoplus_{i} \Lambda/p^{l_i}\oplus\bigoplus_{j} \Lambda/F_j^{n_j}\oplus
\bigoplus_{\substack{m'\equiv m\\\pmod{p-1}}}\bigoplus_{k=1}^{t_{m'}} \Lambda/f_{m'}^{m'_k}
\]
But $E_m$ is finitely generated $\Lambda$-module, so 
\[
\sum_{\substack{m'\equiv m\\\pmod{p-1}}}t_{m'}
\]
must be finite.
This means, $t_m=0$ for all but finitely many $m$.

\begin{cor}
    First suppose that $m\nequiv 0\pmod{p-1}$. Then, for all but finitely many $m$, the following hold:
    \[
    \mathrm{corank} \,H^1(G, \Q_p/\Z_p(m))=r_m=\begin{cases}
        1, & \text{if } m \text{ is odd} \\
        0, & \text{if } m \text{ is even} \\
    \end{cases}
    \]
    and
    \[
    H^2(G, \Q_p/\Z_p(m))=0.
    \]
\end{cor}

\bibliographystyle{alpha}

\url{https://sites.google.com/view/seunghunryu/home}

\medskip

\url{https://sites.google.com/view/taewan-kim}

\end{document}